\documentclass[]{interact}
\usepackage[utf8]{inputenc}
\usepackage{amsmath}
\usepackage{amssymb}
\usepackage{amsthm}
\usepackage{amsfonts}
\usepackage{algorithm}
\usepackage{algpseudocode}
\usepackage{graphicx}
\usepackage{subcaption}
\usepackage{float}
\usepackage{color}
\usepackage{xcolor}
\usepackage{cancel}
\usepackage[normalem]{ulem}
\usepackage{url}

\usepackage{amsmath}
\usepackage{amsfonts}
\usepackage{latexsym}
\usepackage{amssymb}

\newtheorem{thm}{Theorem}[section]
\newtheorem{theorem}[thm]{Theorem}

\newtheorem{lemma}[thm]{Lemma}

\newtheorem{proposition}[thm]{Proposition}

\newtheorem{definition}[thm]{Definition}

\newtheorem{remark}[thm]{Remark}

\newcommand{\beq}{\begin{equation}}
\newcommand{\eeq}{\end{equation}}
\newcommand{\beqa}{\begin{eqnarray}}
\newcommand{\eeqa}{\end{eqnarray}}
\newcommand{\beqas}{\begin{eqnarray*}}
\newcommand{\eeqas}{\end{eqnarray*}}
\newcommand{\ei}{\end{itemize}}

\newcommand{\R}{\mathbb{R}}

\newcommand{\argmin}{\mathrm{argmin}\,}

\title{The Method of Ellipcenters for strongly convex minimization}

\date{}

\begin{document}

\author{
\name{Roger Behling\textsuperscript{a}\thanks{CONTACT Roger Behling. Email: roger.behling@ufsc.br} and
Ramyro Corr\^ea Aquines\textsuperscript{b} and
Eduarda Ferreira Zanatta\textsuperscript{c} and
Vincent Guigues\textsuperscript{b}}
\affil{\textsuperscript{a}Department of Mathematics, Universidade Federal de Santa Catarina (UFSC), Blumenau, Brazil;
\textsuperscript{b}School of Applied Mathematics, Funda\c{c}\~ao Getulio Vargas (FGV EMAp), Rio de Janeiro, Brazil;
\textsuperscript{c}Institute of Mathematics and Statistics of the University of São Paulo (IME-USP), São Paulo, Brazil}
}

\maketitle

\begin{abstract}
This work is about ME, the Method of Ellipcenters. ME was recently introduced by these very authors as a first order accelerated scheme for unconstrained minimization. Its iterates are all centers of ellipses carefully designed to somehow capture ill-conditioning of the underlying optimization problem. In the first article on ME, we were able to prove that it converges with linear rate when the objective function is quadratic and strongly convex, while here we 
extrend ME for strongly convex optimization problems and
derive convergence for any differentiable strongly convex objective. This investigation was inspired by the great performance of ME for quadratic minimization against steepest descent with exact line search, FISTA, Barzilai-Borwein and Conjugate Gradients. The experiments we carry out now, make ME even more attractive from the numerical point of view. On top of that, the theory seems promising for more general settings.
\end{abstract}


% OPTIONAL (mas recomendável)
%\begin{MSCcodes}
%65K05, 90C25, 90C30
%\end{MSCcodes}

\section{Introduction}

In this paper, we propose a new solution method for the problem
\begin{equation}\label{defpb}
x^* = \argmin \{f(x):x \in \mathbb{R}^n\}
\end{equation}
where $f:\mathbb{R}^n \rightarrow \mathbb{R}$ is strongly convex and differentiable.
In \cite{ourME}, the Method of Ellipcenters (ME) was studied
for problems of form
\eqref{defpb} with $f$ quadratic and strongly convex. In this paper, we consider the more general case of an arbitrary strongly convex differentiable function $f$.
We start recalling the definition of a strongly convex function.
{In what follows, we denote by
$\langle x, y \rangle=x^T y$ the
standard Euclidean inner product in
$\mathbb{R}^n$ and by $\|\cdot\|$ the corresponding induced norm 
$\|x\|=\sqrt{\langle x,x\rangle}$.
\begin{definition} Function $f:\mathbb{R}^n \rightarrow \mathbb{R}$
is strongly convex with strong convexity constant $\mu>0$ for norm $\|\cdot\|$ if for every $x,y \in \mathbb{R}^n$, for every $0 \leq t \leq 1$, we have 
\begin{equation}\label{defsc}
f(tx+(1-t)y) \leq tf(x) + (1-t)f(y)-\frac{\mu t(1-t)}{2}\|y-x\|^2.
\end{equation}
\end{definition}

It is well known that if $f:\mathbb{R}^n \rightarrow \mathbb{R}$ is differentiable then $f$ is strongly convex 
with strong convexity constant $\mu>0$ for $\|\cdot\|$
if, and only if, 
$\nabla f$ is $\mu$-strongly monotone, i.e.,
\begin{equation}\label{scprop1}
\mu \|x-y\|^2 \leq 
\langle  \nabla f(y)-\nabla f(x),y-x\rangle
\end{equation}
for all $x,y \in \mathbb{R}^n$.

ME generates a sequence of
points $(x^k)$ ($x^{k+1}$ is computed at iteration $k$) and stops 
with an optimal solution
$x^k$
when
$\nabla f(x^k)=0$.
Otherwise, we are at a non-stationary iterate $x^k$. In this case, an auxiliary point 
$y^k:=x^k-t_k\nabla f(x^k)$ is computed, where $t_k>0$ is such that $f(x^k)=f(y^k)$. Having found $y^k$, we compute the gradient $\nabla f(y^k)$ and if it is a multiple of $\nabla f(x^k)$, ME returns 
$$
x^{k+1}=\frac{1}{2}\Big(x^k + y^k\Big),
$$ as the next iterate. Otherwise, we define the two-dimensional affine space
\begin{equation}\label{defpik}
\Pi_k:=\Big\{x\in \mathbb{R}^n :x=x^k+\mbox{span}\{\nabla f(x^k),\nabla f(y^k)\}\Big\}
\end{equation}
and build an elliptical curve
$E_k$ having the following three properties:
\begin{description}
   \item[ME1] $E_k$ is an ellipse contained in $\Pi_k$;
   \item[ME2] $E_k$ is orthogonal to $\nabla f(x^k)$ at $x^k$;
   \item[ME3] $E_k$ is orthogonal to $\nabla f(y^k)$ at $y^k$.
\end{description}
We will see that the set of centers of ellipses $E_K$ satisfying the above conditions might not be a singleton. It will be proven that these centers form a half-line starting from $\frac{1}{2}(x^k + y^k)$. That said, the next iterate $x^{k+1}$ is then defined as the point with smallest function value regarding $f$ along this half-line.

Getting $x^{k+1}$, such as above, can be seen as an exact generalization of what has been done in \cite{ourME} for strictly convex quadratics.

The whole philosophy of ME, as explained in its debut in \cite{ourME}, is to encapsulate the solution of \eqref{defpb} and to identify ill-conditioning by building appropriate ellipses related to level curves of $f$. Since these ellipses are constructed upon two gradients per iteration, ME is a first order method. Nevertheless, ME becomes Newton in two-dimensional spaces when $f$ is quadratic. The Newtonian flavor was captured even in higher dimensions in the experiments in \cite{ourME} and there is a reason why our two-dimensional ellipses will provide much more than just double the gain of the one-dimensional line search iteration of the gradient method. It is known that the zigzagging mesh generated by the gradient method steps converges to a two dimensional subspace (which is related to smallest and largest eigenvalues of Hessians for twice differentiable functions). Bearing this in mind as a motivation, we were confident when developing ME in the first paper and are still motivated now by its extension to strongly convex minimization. The preliminary experiments are also quite exciting 
{since ME is quicker
than all three competitors
that we consider (gradient descent with
exact line search, Nesterov accelerated gradient \cite{nesterov1983}, and the
method from \cite{MEbellocruz}) to minimize a
 $\ell_2$-regularized logistic loss function.}

\iffalse
and the details of the computation of the center of the ellipse was given for $f$ quadratic.
In this paper, we consider the general strongly convex case.

Accelerating the gradient method with first-order tools has long been a popular subject of research in the field of Continuous Optimization, see for instance \cite{grimmer2023optimal,lan2015bundle,mishchenko2020adaptive,nesterov2015universal,renegar2022simple} and \cite{zhou2024adabb}. Perhaps one of the most famous of these methods is given in the paper \cite{nesterov1983} by Y. Nesterov (fine tuning
of this algorithm is discussed in \cite{gonzagakaras}). He embedded inertia in the gradient method and was able to derive an iteration with best possible complexity for a first-order method, \textit{i.e.}, a method using first derivatives only. Another first-order algorithm that enjoys optimal complexity in theory is the Ellipsoid method \cite{Khachiyan1979}. Nevertheless, in practice, the Ellipsoid method, which is different from what we are proposing here, is not attractive.
\fi

Our paper is organized as follows.
In Section \ref{sec:algo}, we rewrite ME in a rather direct and computational manner.
In Section \ref{sec:geometry} we study properties of the
method and also show that descent is ensured at every iteration.
In Section \ref{sec:conv}, we prove convergence of ME for the minimization of strongly convex differentiable functions
while in Section
\ref{sec:num} we provide the results of preliminary numerical experiments
comparing ME with other popular methods for convex minimization.
Finally, concluding remarks can be found in the final
Section \ref{sec:conc}.

\section{Algorithm}\label{sec:algo}

{In this section, we present an algorithm that generates iterates $(x^k)$ as explained in the introduction, expressing every new iterate $x^{k+1}$ at iteration $k$ as the center of an ellipse
satisfying ME1, ME2, and ME3. At every iteration, the set of centers of the ellipses satisfying ME1, ME2, and ME3 is a half-line (the proof of this property is shown in Section 
\ref{sec:geometry}) and  we minimize the
function on this half-line.
The existence of the minimizer (denoted by
$v_k$ in the algorithm) of
function $f$ on this half-line follows
from the continuity and coercivity of
$f$ on the closed half-line.}

It will be shown in Section \ref{sec:conv} that this
algorithm converges to $x^*$, an optimal solution to problem \eqref{defpb}.

\noindent\rule[0.5ex]{1\columnwidth}{1pt}
	Method of ellipcenters for strongly convex functions with line search.\\
    \noindent\rule[0.5ex]{1\columnwidth}{1pt}
    {\bf Inputs:} Initial point $x^1 \in \mathbb{R}^n$, $k=1$.\\
    
\noindent {\textbf{Step 1.}} If $\nabla f(x^k)=0$ stop and return $x^k$. Otherwise, go to Step 2.\\

\noindent {\textbf{Step 2.}} Compute
step $t_k$ such that
\begin{equation}\label{defyk}
y^k = x^k - t_k \nabla f(x^k)
\end{equation}
and $f(y^k)=f(x^k)$.

\noindent {\textbf{Step 3.}} If $\nabla f(y^k)$ and $\nabla f(x^k)$
are linearly independent compute
\begin{align}
\cos \theta_k &= \frac{\langle x^k-y^k,-\nabla f(y^k)\rangle}{\|x^k-y^k\|\|\nabla f(y^k)\|}, \label{costk}\\
\sin \theta_k &= \sqrt{1-(\cos\theta_k)^2},\\
w^k & = -\nabla f(y^k) + \frac{\langle x^k-y^k,\nabla f(y^k)\rangle}{\|x^k-y^k\|^2}(x^k-y^k),\\
d^k&= \frac{w^k}{\|w^k\|}-\frac{\sin \theta_k}{2 \cos \theta_k}\frac{x^k-y^k}{\|x^k-y^k\|}. \label{formdk}
\end{align}
$\hspace*{1.9cm}$Minimize $f$ on a half-line
with origin $(1/2)(x^k+y^k)$ and direction
$d^k$ to find $v_k \geq 0$ such 
$\hspace*{1.9cm}$that
\begin{equation}\label{semilinealgo}
v_k:= \mbox{argmin}\left\{f\left(\frac{1}{2}\left(x^k + y^k\right)+v d^k \right) : v \geq 0  \right\}
\end{equation}
$\hspace*{1.9cm}$and set
\begin{equation}\label{defxcase1}
x^{k+1}=
\frac{1}{2}\left(x^k + y^k\right)+v_k d^k;
\end{equation}
$\hspace*{1.5cm}$else if $\nabla f(y^k)$ and $\nabla f(x^k)$
are linearly dependent compute
\begin{equation}\label{defxcase2}
x^{k+1} = \frac{1}{2}\left(x^k + y^k\right).
\end{equation}
$\hspace*{1.5cm}$end if\\
$\hspace*{1.5cm}$ Perform $k \leftarrow k+1$ and go to Step 1.\\
\noindent\rule[0.5ex]{1\columnwidth}{1pt}

{We now comment on two variants of ME that will be
tested in our numerical experiments, together with the original
ME given above. Convergence and complexity
analysis of ME given in Theorems \ref{convME1} and \ref{compME} 
also applies to these variants of ME.}

{The first variant of ME, denoted by 
ME-midpoint, computes at every iteration
$x^{k+1}$ using the update $x^{k+1}=z^k$
where by definition
$z^k$ is given by
$z^k:=(1/2)(x^k+y^k)$ and
where $y^k$ is computed as explained in ME.}

{The second variant of ME, denoted by 
ME-linesearch, 
still computes $y^k$ as in ME but instead of
minimizing $f$ on the half-line 
$\{z^k + v d^k : v \geq 0\}$ 
performs a line search
on this half-line 
to try to find a point where
the value of $f$ is strictly less than the value
$f(z^k)$
of $f$ at $z^k$.
We initialize the step size $v=\mbox{vmax}$
for some arbitrary positive value $\mbox{vmax}$ (say 10)
and specify a maximum number of iterations maxiter
for the line search. At iteration $k$  of this line search, we evaluate $f$
at  $z^k+(\mbox{vmax}/2^{k-1})d^k$ and accepts
this point (which is $x^{k+1}$) if the value of $f$ at this point is
strictly less than $f(z^k)$.
If in maxiter iterations, we have not found a
point satisfying this condition, the line search
returns $z^k$ and $x^{k+1}=z^k$. The pseudo-code of this line search is given below.}  
\if{
i=1
v=vmax
fz = $f(z^k)$  
ftest = f(z^k+v*d^k)

while i <= iterMax && ftest >= fz
        v /= 2
        ftest = f(z^k .+ v .* d^k)
        i += 1
    end
    if ftest < fz
        return z^k .+ v .* d^k
    else
        return z^k
    end
}\fi
}

\noindent\rule[0.5ex]{1\columnwidth}{1pt}
ME-linesearch.\\
\noindent\rule[0.5ex]{1\columnwidth}{1pt}
{\bf Inputs:} Point $z^k\in\mathbb{R}^n$, direction $d^k$,
$\mathrm{vmax}>0$, and $\mathrm{maxiter}\geq 1$.\\

\noindent {\textbf{Step 1.}} Set
\(
i\gets 1,
v\gets \mathrm{vmax},
\)
and compute
\(
f_z\gets f(z^k),
f_{\mathrm{test}}\gets f(z^k+vd^k).
\)

\noindent {\textbf{Step 2.}} While
\(
i\leq \mathrm{maxiter}
\)
and
\(
f_{\mathrm{test}}\geq f_z,
\)
perform
\begin{align*}
&v \gets \frac{v}{2},\\
&f_{\mathrm{test}} \gets f(z^k+vd^k),\\
&i\gets i+1.
\end{align*}
\quad \qquad \quad End while.

\noindent {\textbf{Step 3.}} If
\(
f_{\mathrm{test}}<f_z,
\)
return
\(
z^k+vd^k.
\)
Otherwise, return $z^k$.\\
\noindent\rule[0.5ex]{1\columnwidth}{1pt}

{We now make a few comments on
Step 3 of the algorithm.
When $\nabla f(y^k)$
and $\nabla f(x^k)$ are linearly
independent, we have 
$\nabla f(x^k) \neq 0$, 
$\nabla f(y^k) \neq 0$ (otherwise,
if we had $\nabla f(y^k)=0$ then we would
have that $y^k$ is the minimizer of $f$ but
since $f(x^k)=f(y^k)=f(x^*)$ we would also
have that $x^k$ is the minimizer of $f$ which is a contradiction with $\nabla f(x^k) \neq 0$), and
$x^k \neq y^k$. This implies that
$\cos \theta_k$ given by \eqref{costk}
is well defined.
Moreover, by strong convexity of $f$, we have
$$
f(x^k) \geq  f(y^k) + \langle \nabla f(y^k),x^k-y^k\rangle + \frac{\mu}{2}\|x^k-y^k\|^2 
$$
which implies, using $f(x^k)=f(y^k)$, that
$$
-\langle \nabla f(y^k),x^k-y^k\rangle 
\geq \frac{\mu}{2}\|x^k-y^k\|^2 >0 
$$
and therefore 
\begin{equation}\label{costheta}
    \frac{\langle x^k-y^k,-\nabla f(y^k)\rangle}{\|x^k-y^k\| \| \nabla f(y^k)\|}=\cos \theta_k>0. 
\end{equation}

By linear independence of
$\nabla f(y^k)$ and $\nabla f(x^k)$,
we also have $\sin \theta_k>0$ and $w^k \neq 0$. Therefore, direction $d^k$ in
\eqref{formdk} is well defined.

\section{Properties of the algorithm}
\label{sec:geometry}

We now make a few comments about ME and study its properties. First, if $\nabla f(x^k)=0$ then 
$x^k$ is an optimal solution of
\eqref{defpb} and the algorithm stops.
Otherwise, the algorithm computes
an auxiliary point $y^k$ in the same level
set as $x^k$, i.e., satisfying 
\begin{equation}\label{levcurve}
f(y^k)=f(x^k).
\end{equation}
Lemma \ref{uniqueness} shows that there is
at most one such $y^k$ while Lemma 
\ref{existence} establishes that there is at least one
such $y^k$, ensuring existence and uniqueness
of $y^k$ satisfying \eqref{levcurve}.

\begin{lemma}\label{uniqueness} Let $f: \mathbb{R}^n \rightarrow \mathbb{R}$ be a strongly convex function
with strong convexity constant $\mu$
for some norm $\|\cdot\|$. Given a line $L$, there
are at most two different points $x,y \in L$ where $f$ has the same value,
i.e., such that $f(x)=f(y)$.
\end{lemma}
\begin{proof} Assume by contradiction that there are three
different points $x, y, z$ in a line $L$ such that $f(x)=f(y)=f(z)$.
Denote by $\tilde f$ the common value of $f$ at these points
and without loss of generality assume that $y$ belongs
to the segment $[x,z]=\{tx+(1-t)z: 0 \leq t \leq 1\}$, i.e.,
$y=tx+(1-t)z$ for some $0<t<1$. Then
$$
\tilde f=f(y)=f(tx+(1-t)z) \leq 
tf(x)+(1-t)f(z)-\frac{\mu t(1-t)}{2}\|x-z\|^2
=\tilde f -\frac{\mu t(1-t)}{2}\|x-z\|^2
$$
which implies $x=z$ (since $\mu>0$ and
$1>t>0$) and yields the desired contradiction. 
\end{proof}

\begin{lemma}\label{existence}
Let $f: \mathbb{R}^n \rightarrow \mathbb{R}$ be a differentiable and 
strongly convex function with strong convexity constant $\mu$ for some norm $\|\cdot\|$. Let $x \in \mathbb{R}^n$ such that
$\nabla f(x) \neq 0$. Then
there exists a unique $t>0$
such that
$f(x-t\nabla f(x))=f(x)$.
\end{lemma}
\begin{proof}
Let $g(t)=f(x-t\nabla f(x))$.
Since $g'(0)<0$, there is
$t_0>0$ such that
$g(t_0)<g(0)$. 
{Next, observe that
$
 \|x-t\nabla f(x)\|\to+\infty$
 as $t\to+\infty$
because $\nabla f(x)\neq0$.
Since $f$ is coercive, 
this implies
$\lim_{t \rightarrow +\infty} g(t)=+\infty$.
}
Since $f$
is strongly convex on $\mathbb{R}^n$, it is continuous
on $\mathbb{R}^n$ implying
that $g$ is continuous and that
there is $t$ satisfying
$t_0<t<+\infty$ such that
$g(t)=g(0)$ or equivalently such that $f(x-t\nabla f(x))=f(x)$. By Lemma \ref{uniqueness}, there cannot be more than one such $t>t_0$, which achieves the statement of the lemma.
\end{proof}

We will show in the next section that the set of centers
of ellipses satisfying ME1, ME2, and ME3
is the half-line
$$
L_k:=\left\{\frac{1}{2}\Big(x^k+y^k\Big)+v d^k : v \geq 0\right\}
$$
where $d^k$ is computed in the algorithm.
Therefore, ME not only finds
a point which is the center of an ellipse
satisfying ME1, ME2, and ME3 but also
ensuring a sufficient decrease of the objective on the half-line $L_k$.
As a special case of a line search in $L_k$, we can set $x^{k+1}$ as the minimizer of $f$ in $L_k$.

\subsection{Computation of the set of centers of ellipses satisfying ME1, ME2, and ME3}

Since the center of ellipse $E_k$ is in
$\Pi_k$, we can reparametrize in 
$\mathbb{R}^2$
the computation of centers of ellipses
satisfying ME1, ME2, and ME3, see Figure \ref{fig:ellipses}.
To alleviate notation we write
$x$ instead of $x^k$, $y$ instead of
$y^k$ and assume that
$x$ has coordinates
$x=(\lambda,0)$ with $\lambda\geq 0$ and
$y=(0,0)$ is at the origin.
We also consider an arbitrary point
$z=(m,n)$ on the half-line 
$$
\{y-t\nabla f(y):t\geq 0\},
$$
i.e., $z \in \Pi_k$ is of form
$y-t\nabla f(y)$ for some $t \geq 0$
and $m,n \geq 0$. Observe that $z$ moves
along a half-line with slope
given by the angle $\theta$ such that
$$
\tan \theta=n/m,
$$
determined by $x, y$, and $\nabla f(y)$. In fact, $m=\| z-y \|\cos \theta \stackrel{\eqref{costheta}}{>}0$. The reader can refer to Figure
\ref{fig:image6} for a graphical illustration of the notation we have just introduced (in this figure, we have represented the angle $\theta$, $x$, $y$, $z$, $\nabla f(y)$,
$\nabla f(x)$, and the vector $w:=w_k/\|w_k\|$ in ME).

\begin{figure}[h!]
    \centering
    \includegraphics[width=0.8\linewidth]{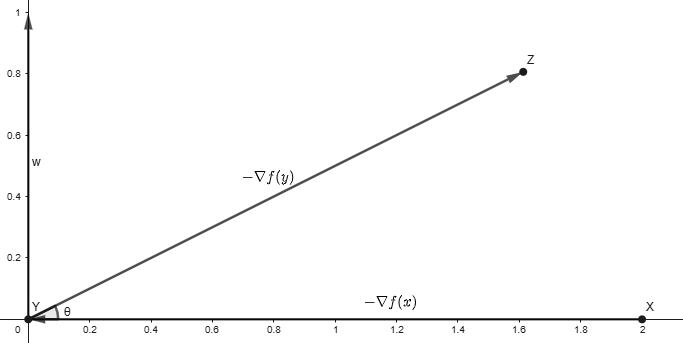}
    \caption{Setup for points $x$, $y$, and $z$.}
    \label{fig:image6}
\end{figure}
We are looking for an ellipse (and its center) of form
\[
\varphi(\alpha,\beta)
=\frac12\begin{bmatrix}\alpha&\beta\end{bmatrix}\!
\begin{bmatrix}1&b\\ b&a\end{bmatrix}
\!\begin{bmatrix}\alpha\\ \beta\end{bmatrix}
+\begin{bmatrix}c&d\end{bmatrix}\!\begin{bmatrix}\alpha\\ \beta\end{bmatrix}.
\]
This ellipse has to satisfy
\begin{enumerate}
    \item[(P1)] $\varphi(x)=\varphi(y)=\varphi(z)=0$;
    \item[(P2)] $y-x$ (which is a positive multiple of $-\nabla f(x)$) is a positive multiple of $-\nabla\varphi(x)$ and $z-y$ (which is a positive multiple of  $-\nabla f(y)$) is a positive multiple of $-\nabla\varphi(y)$.
\end{enumerate}
Property (P1) ensures that the ellipse
passes through $x$, $y$, and $z$ while (P2) ensures that
ME2 and ME3 are satisfied. From our computations, the ellipse will also be contained in $\Pi_k$ meaning that
ME1 is also satisfied.

\begin{figure}[h!]
    \centering
    
    % --- First figure ---
    \begin{subfigure}[t]{0.45\textwidth}
        \centering
        \includegraphics[scale=0.45]{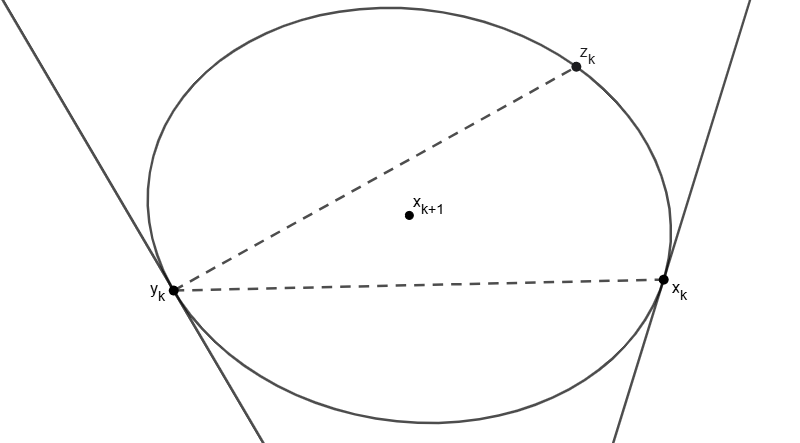}
        \caption{$z$ closer to $y$, yielding a wider ellipse}
        \label{fig:ellipse1}
    \end{subfigure}
    \hfill
    % --- Second figure ---
    \begin{subfigure}[t]{0.45\textwidth}
        \centering
        \includegraphics[scale=0.5]{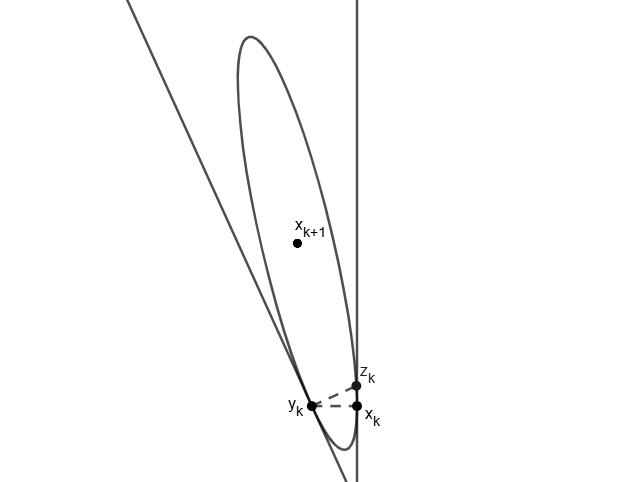}
        \caption{$z$ farther from $y$, yielding an elongated ellipse}
        \label{fig:ellipse2}
    \end{subfigure}

    \caption{Ellipses for different choices of $z$ along $\{y-\theta \nabla f(y) :  \theta \geq 0\}$}
    \label{fig:ellipses}
\end{figure}

\begin{proposition}[Values of $a$, $b$, $c$, and $d$] The values of $a, b, c, d$ satisfying (P1) and (P2) are $$a=\left(\frac{\lambda}{m} - 1\right)\left( \big(\cot \theta \big)^2 + 1 \right), \quad b=\frac{\tan \theta}{2}, \quad c=-\lambda/2, \quad d=- \frac{\lambda}{2}\tan \theta.$$
\end{proposition}

\begin{proof}
From $\varphi(x)=0$,
\[
0=\tfrac12\lambda^2 + c\,\lambda \quad\Longrightarrow\quad c=-\frac{\lambda}{2}.
\] Normality at $x$ gives $\nabla\varphi(x)=(\lambda+c,\; b\lambda+d)$ and $y-x=(-\lambda,0)$ is a positive multiple of $-\nabla\varphi(x)$. Hence $b\lambda+d=0$ and
\(
d=-b\lambda.
\)
Moreover, normality at $y$ gives $\nabla\varphi(y)=(c,d)=(-\lambda/2,\,-b\lambda)$ and $z-y=(m,n)$ must be a positive multiple of $-\nabla\varphi(y)=(\lambda/2,\; b\lambda)$. Thus there exists $t>0$ with
\(
m=t\,\frac{\lambda}{2},n=t\,b\lambda
\quad\Longrightarrow\quad 
\,b=\frac{n}{2m}
\).
Now, using $c=-\lambda/2$ and $d=-b \lambda$,
\[
0=\varphi(m,n)=\tfrac12\big(m^2+2bmn+an^2\big) -(\lambda/2)m - b \lambda n.
\]
Multiplying by $2$ and rearranging:
\[
a\,n^2=\lambda m + 2b \lambda n - m^2 - 2bmn.
\]
Substituting $b=\frac{n}{2m}$ yields
\[
a\,n^2=\lambda m + \lambda\frac{n^2}{m} - m^2 - n^2,
\qquad\Longrightarrow\qquad
a=\frac{(\lambda-m)(m^2+n^2)}{m\,n^2}=\left(\frac{\lambda}{m} - 1\right)\left( \big(\cot \theta \big)^2 + 1 \right), 
\]
as desired.
\end{proof}

\begin{proposition}[Fitted conic is an ellipse]\label{prop:fitconic}

The fitted conic satisfying (P1) and (P2) is an ellipse if, and only if, $$0<m<\frac{4 \lambda \big(1+\big(\tan{\theta} \big)^2\big)}{\big( \big(\tan{\theta} \big) ^2+2 \big)^2 },$$

where $\tan \theta=\frac{n}{m}$.

\end{proposition}

\begin{proof} The conic $\{\varphi=0\}$ is an ellipse if and only if $a>b^2$. Using $b^2=\frac{n^2}{4m^2}$ and $a=\frac{(\lambda-m)(m^2+n^2)}{m\,n^2}$, $m$ must be positive, and:
\[
\frac{(\lambda-m)(m^2+n^2)}{m\,n^2} \;>\; \frac{n^2}{4m^2}
\quad\Longleftrightarrow\quad
4m(\lambda-m)(m^2+n^2) \;>\; n^4.
\]
With the parametrization $n=m\tan\theta $, this becomes
\[
4m(\lambda-m)m^2(1+\big(\tan\theta \big)^2) > \big(\tan\theta \big)^4 m^4
\;\;\Longleftrightarrow\;\;
m < \frac{4 \lambda \big(1+\big(\tan\theta \big)^2\big)}{\big(\big(\tan\theta \big)^2+2\big)^2}.
\]

\end{proof}

\begin{proposition}[Center of an admissible ellipse]
\label{prop:center_formula}
The center $(u,v)$ of the ellipse $\{\varphi=0\}$ is given by
\[
u=\frac{\lambda}{2} \frac{a-2b^2}{a-b^2},\qquad v=\frac{\lambda}{2}\frac{b}{a-b^2}.
\]
\end{proposition}

\begin{proof}
The center solves $\nabla\varphi(u,v)=0$, i.e.
\(
\begin{bmatrix}1&b\\ b&a\end{bmatrix}\!\begin{bmatrix}u\\ v\end{bmatrix}
=-\begin{bmatrix}c\\ d\end{bmatrix} = \begin{bmatrix}\lambda/2\\ n\lambda/2m\end{bmatrix}.
\)
Solving this $2\times2$ linear system yields the stated formulas.
\end{proof}

\begin{proposition}[Line of centers]
\label{prop:line_centers}
The set of centers of all admissible ellipses is contained in the line
\[
u \;=\; \lambda/2 - \Big(\frac{\tan \theta}{2}\Big)\,v .
\]
\end{proposition}

\begin{proof}
From Proposition~\ref{prop:center_formula}, using $b=\frac{n}{2m}$ and $c=-\lambda/2$, we have
\(
u=\lambda/2-bv=\lambda/2 - \Big(\frac{\tan \theta}{2}\Big)\,v.
\)
\end{proof}

\begin{lemma}[Admissible {half-line} of centers]
\label{lem:segment_centers}
The admissible centers form the half-line
\[
\mathcal{C} \;=\; \big\{\, (u,v)\in\R^2 \;:\; u=\frac\lambda2-\tfrac{\tan \theta}{2}v,\;\; v\in(0,+\infty)\,\big\}.
\]

Moreover, the unbounded part of the half-line corresponds to the degeneracy $a\downarrow b^2$ of $\{\varphi=0\}$, whereas the finite endpoint $(\lambda/2,0)$ is obtained in the limiting regime $a\to+\infty$.

\end{lemma}

\begin{proof}
Fix $b=\frac{n}{2m}$ and view $v$ as a function of the shape parameter $a$:
\[
v(a)=\frac{\lambda}{2}\frac{b}{a-b^2}.
\]
Then
\(
v'(a)=\frac{\lambda}{2}\frac{-b}{(a-b^2)^2}
<0,
\)
so the center moves monotonically along the line as $a$ increases.\\
% with the limits
% \[
% \lim_{a\downarrow b^2}u(a)= -\infty,\qquad \lim_{a\uparrow\infty}u(a)=1.
% \]

Furthermore,
$
\lim_{a\to+\infty} v(a)=0
$ and
$
\lim_{a\downarrow b^2} v(a)=+\infty.
$
Since \(v(a)\) is continuous and strictly decreasing for \(a>b^2\), it follows that
\[
v(a)\in(0,+\infty).
\]
Therefore, the centers of nondegenerate admissible ellipses are parametrized by \(v>0\). In particular, the limit \(v\to+\infty\) corresponds to the degeneracy \(a\downarrow b^2\), whereas the finite endpoint
\[
\left(\frac{\lambda}{2},0\right)
\]
is reached only in the limiting regime \(a\to+\infty\).

\end{proof}

We now express the coordinates of the centers in $\mathbb{R}^n$.
Consider an orthonormal basis of the subspace $\Pi_k$ 
whose first basis vector is 
$$
\frac{x^k-y^k}{\|x^k-y^k\|}
$$
(observe that $x^k \neq y^k$, otherwise 
$\nabla f(x^k)=0$
which would imply that
$x^k$ is an optimal solution).
The second basis vector is chosen to be
$$
\frac{w^k}{\|w^k\|}
$$
where
$$
w^k=-\nabla f(y^k) - \|\nabla f(y^k)\|\frac{\cos \theta_k}{\|x^k-y^k\|}(x^k-y^k)=
-\nabla f(y^k) + \frac{\langle x^k-y^k,\nabla f(y^k)\rangle}{\|x^k-y^k\|^2}(x^k-y^k).
$$
It is immediate to check that
$w^k$ and the first basis vector
$(x^k-y^k)/\|x^k-y^k\|$ are orthogonal. Moreover, $$\langle w^k, -\nabla f(y^k) \rangle = \| \nabla f(y^k) \|^2 - \frac{\langle x^k-y^k, \nabla f(y^k) \rangle^2}{\| x^k - y^k \|^2} \geq 0,$$ by Cauchy-Schwarz inequality. This implies that the angle between $w^k$ and $-\nabla f(y^k)$ is non-obtuse, which is {in accordance} with Figure \ref{fig:image6}.
The origin is taken to be $y^k$.
With this definition,
by {Lemma \ref{lem:segment_centers}}, a center of an ellipse
$E_k$ satisfying ME1, ME2, and {ME3}, is of form
$$
y^k + 
\Big(\frac{\lambda}{2}-\frac{\tan \theta}{2}v_k\Big)
\frac{x^k-y^k}{\|x^k-y^k\|}
+ v_k \frac{w^k}{\|w^k\|}=\frac{1}{2}\left(x^k + y^k\right)+v_k d^k
$$
where $v_k \geq 0$.
In the equality above,
we have used the fact that $\lambda=\|x^k-y^k\|$
and the definition of $d^k$ given  in ME.
% \subsection{Rules to select a center on the line}
% \label{subsec:rules}

% Let $\mathcal{L}_{\text{cent}}=\{c_0+\tau d:\tau\in(\tau_-,\tau_+)\}$ be an affine parametrization of the line in Proposition~\ref{prop:line_centers} restricted to the admissible segment of Lemma~\ref{lem:segment_centers}.
% We use the following selection rules.

% \begin{itemize}
% \item[\textnormal{(S1)}] \textbf{Exact minimizer on the line.} 
% Choose $x^{k+1}\in\arg\min\{\,f(c):\,c\in\mathcal{L}_{\text{cent}}\}$.
% This yields $f(x^{k+1})\le f(z^k)$ since $z^k$ belongs to the line of centers (Section~\ref{sec:convergence}).
% \end{itemize}

\begin{figure}
    \centering
    \includegraphics[width=0.75\linewidth]{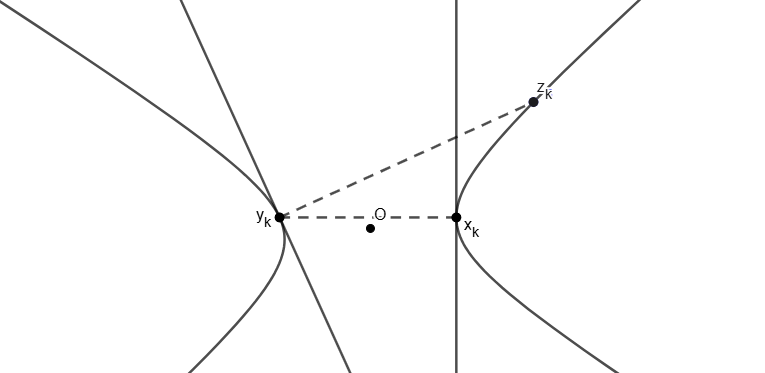}
    \caption{Fitted conic is a hyperbola when $v \not \in (v_-,v_+)$}
    \label{fig:placeholder}
\end{figure}

{Observe that ME from Section \ref{sec:algo} 
allows the step $v_k$ given by \eqref{semilinealgo}
to be zero. In this limiting case, we can still 
consider the corresponding point $(1/2)(x^k+y^k)$ 
as the center of a degenerate ellipse which is the segment
$[x^k,y^k]$.}

\section{Convergence}\label{sec:conv}

\subsection{Global convergence for strongly convex functions}

For any integer $k \geq 1$, let $x^k$ be the $k$-th iterate and $t_k>0$ the unique real number such that $f(x^k-t_k \nabla f(x^k))=f(x^k)$. Let $z^k := x^k-\frac{t_k}{2} \nabla f(x^k)$. From the previous section, we know that, in the $k$-th iterate, the half-line of the centers of the ellipses passes through $z^k$. Since $x^{k+1}$ is the minimizer of $f$ along the half-line of centers, we obtain that $f(x^{k+1}) \leq f(z^k) < f(x^k)$. 

\begin{lemma}\label{lemcompactdom}
Let $f:\mathbb{R}^n \rightarrow \mathbb{R}$
be a strongly convex function. Then the level set
\begin{equation}\label{defL}
\mathcal{L}=\{ x \in \mathbb{R}^n|f(x) \leq f(x^1) \}
\end{equation}
for $x^1 \in \mathbb{R}^n$
is compact.
\end{lemma}
\begin{proof}
Assume that 
$\mathcal{L}$ is not bounded. Then we can find
a sequence 
$(x^k)$ such that
$\lim_{k \rightarrow +\infty} \|x^k\| = +\infty$ and $f(x^k) \leq f(x^1)$. Moreover, by strong convexity of $f$, we have
$$
f(x^k) \geq f(x^*) + \frac{\mu}{2} \| x^k-x^*\|^2,
$$
which implies
$\lim_{k \rightarrow +\infty} f(x^k)=+\infty$. This is in contradiction with $f(x^k) \leq f(x^1)$.
Therefore $\mathcal{L}$ is bounded. Being the sublevel set of
a convex function, $\mathcal{L}$ is closed.
It follows that $\mathcal{L}$ is compact.
\end{proof}

\begin{proposition}\label{prop:main}
    For all $\eta>0$, there exists $\tau>0$ such that for any $x$ satisfying $f(x) \leq f(x^1)$ and $||\nabla f(x) || \geq \eta$, we have 
    $$
    f(x-t\nabla f(x)) \leq f(x) - \frac{\eta^2}{2}t 
    $$ for any $t \in [0, \tau]$.
\end{proposition}
\begin{proof} Let
\begin{equation}\label{defLbis}
\mathcal{L}=\{ x \in \mathbb{R}^n|f(x) \leq f(x^1) \}
\end{equation}
and define $\phi: \mathcal{L} \times [0,1] \rightarrow \mathbb{R}$ by $\phi(x,t)=f(x-t \nabla f(x))$. 

We have that $\phi$ is differentiable with 
$$\frac{\partial \phi}{\partial t}(x,t)= - \langle \nabla f(x-t \nabla f(x)), \nabla f(x) \rangle.
$$
In particular, $\frac{\partial \phi}{\partial t}(x,0)=-||\nabla f(x)||^2$.

By Lemma \ref{lemcompactdom}, the domain of $\phi$ is compact.
Together with the fact that
$\frac{\partial \phi}{\partial t}$ is continuous ({which follows from the differentiability and convexity of $f$ on the open set $\mathbb{R}^n$}), this implies that  $\frac{\partial \phi}{\partial t}$ is uniformly continuous on its domain. Hence, for every $\eta>0$, there exists $\tau>0$ such that if $|u-v|<\tau$ the inequality 
$$\left|\frac{\partial \phi}{\partial t}(x,u)-\frac{\partial \phi}{\partial t}(x,v)\right| \leq \frac{\eta^2}{2}$$
holds. We then have $$\frac{\partial \phi}{\partial t}(x,s)+||\nabla f(x)||^2 \leq \left|\frac{\partial \phi}{\partial t}(x,s)+|| \nabla f(x)||^2\right|=\left|\frac{\partial \phi}{\partial t}(x,s)-\frac{\partial \phi}{\partial t}(x,0)\right| \leq \frac{\eta^2}{2}$$
for any $s \in [0, \tau]$. It follows  that $$\frac{\partial \phi}{\partial t}(x,s) \leq \frac{\eta^2}{2}-||\nabla f(x)||^2, \quad \forall s \in [0,\tau].$$

Thus, if $||\nabla f(x)||\geq\eta$, we can find a $\tau>0$ such that the above inequality holds for all $s \in [0,\tau]$ and $x \in \mathcal{L}$, by the uniform continuity of $\frac{\partial \phi}{\partial t}$. But then $\frac{\eta^2}{2}-||\nabla f(x)||^2 \leq \frac{\eta^2}{2}-\eta^2=-\frac{\eta^2}{2}$, meaning that $$\frac{\partial \phi}{\partial t}(x,s) \leq -\frac{\eta^2}{2}, \;\; \forall s \in [0,\tau], \forall  x \in \mathcal{L}.
$$
Therefore, for any $t \leq \tau$, $$\phi(x,t)-\phi(x,0)=\int_0^t \frac{\partial \phi}{\partial t}(x,s)ds \leq \int_0^t -\frac{\eta^2}{2}=-t\frac{\eta^2}{2},$$
which implies 
$$f(x-t \nabla f(x)) \leq f(x) - \frac{\eta^2}{2} t
$$ for all $t \in [0,\tau]$ and $x \in \mathcal{L}$, as desired.
\end{proof}

\begin{theorem}\label{convME1}
Let $(x^k)$ be the sequence generated
by ME. Then 
$\lim_{k \rightarrow +\infty} x^k=x^*$.
\end{theorem}
\begin{proof} First observe that by strong convexity of $f$, we have
\begin{align}
f(z^k) & =  f\Big(x^k-\frac{t_k}{2}\nabla f(x^k)\Big)\\
& =  f\Big(\frac{1}{2}\Big(x^k- t_k\nabla f(x^k)\Big) + \frac{1}{2}x^k\Big)\\
& \stackrel{\eqref{defsc}}{\leq} \frac{1}{2}f\Big(x^k- t_k\nabla f(x^k)\Big) + \frac{1}{2}f\Big(x^k\Big) -\frac{\mu t_k^2}{8} || \nabla f(x^k)||^2 \\
&= f(x^k)-\frac{\mu t_k^2}{8} || \nabla f(x^k)||^2. \label{finalzinfx}
\end{align}
Next notice that by definition of 
$x^{k+1}$ we have
$f(x^{k+1}) \leq f(z^k)$
while \eqref{finalzinfx} shows in particular that
$f(z^k) \le f(x^k)$ and we therefore have
\begin{equation}
f(x^{k+1}) \leq f(z^k) \leq f(x^k).
\end{equation}
Since the sequence $(f(x^k))$ is 
nonincreasing and bounded below by
$f(x^*)$ it converges to some
$\ell \in [f(x^*),f(x^0)]$. 

Let $S_n=\sum_{k=1}^n f(x^k)-f(z^k)$.
Since $f(x^k)-f(z^k) \geq 0$, sequence
$(S_n)$ is nondecreasing and it is upper bounded:
$$
S_n\leq \sum_{k=1}^n f(x^k)-f(x^{k+1})=
f(x^0)-f(x^{n+1})\leq f(x^1)-f(x^*).
$$
It follows that $S_n$ converges, so that
\begin{equation}\label{convzxz}
\lim_{k \rightarrow +\infty} f(x^k)-f(z^k)=0.
\end{equation}
Combining \eqref{finalzinfx} and \eqref{convzxz} we obtain
\begin{equation}
0 \leq \lim_{k \rightarrow \infty} \frac{\mu t_k^2}{8}||\nabla f(x^k)||^2 \leq \lim_{k \rightarrow \infty} (f(x^k)-f(z^k))=0
\end{equation}
which gives
\begin{equation}\label{tngradg0}
 \lim_{k \rightarrow \infty} \frac{\mu t_k^2}{8}||\nabla f(x^k)||^2 = 0.
 \end{equation}

Assume by contradiction that
$\lim_{n \rightarrow \infty}||\nabla f(x^n)|| \neq 0$. Then there exists a subsequence $\{k_n\}_{n=1}^{\infty}$ and a positive constant $\eta$ such that $||\nabla f(x^{k_n})|| \geq \eta$ for all $n \in \mathbb{N}$. Since $\{x^{k_n}\}_{n=1}^{\infty} \subset \mathcal{L}$ where $\mathcal{L}$ is given by \eqref{defLbis}, from Proposition ~\ref{prop:main} there exists $\tau>0$ such that $$f(x^{k_n}-t \nabla f(x^{k_n})) \leq f(x^{k_n})- \frac{\eta^2}{2}t$$
for all $t \in [0,\tau]$ and $n \in \mathbb{N}$.

Thus, if $t_{k_n}<\tau$ for some $n$, we would have $$f(x^{k_n})=f(x^{k_n}-t_{k_n} \nabla f(x^{k_n})) \leq f(x^{k_n})-\frac{\eta^2}{2} t_{k_n}$$ implying that  $\eta^2 t_{k_n} \leq0$, a contradiction (since $\eta>0$, $t_{k_n}>0$). Therefore, $t_{k_n} \geq \tau$, which implies that $ t_{k_n}^2 || \nabla f(x^{k_n})||^2 \geq \tau^2 \eta^2 >0$, a contradiction with \eqref{tngradg0}. 

Therefore, 
\begin{equation}\label{limgradz}
\lim_{n \rightarrow \infty} \| \nabla f(x^n)\|=0.
\end{equation}
It follows, using the strong convexity of $f$ and {Cauchy-Schwarz} inequality, that
$$
\mu \|x^k-x^*\|^2 
\stackrel{\eqref{scprop1}}{\leq} \langle 
\nabla f(x^k)-\nabla f(x^*),x^k-x^* 
\rangle = \langle 
\nabla f(x^k),x^k-x^* 
\rangle  \leq \|\nabla f(x^k)\| \|x^k-x^*\|
$$
from which we deduce
\begin{equation} 
\|x^k-x^*\| \leq (1/\mu) \|\nabla f(x^k)\|
\end{equation}
which, combined with \eqref{limgradz}, gives
$\lim_{k \rightarrow +\infty} x^k=x^*;$
\end{proof}

\subsection{Linear convergence with rate $\eta=1-\frac{\mu^2}{L^2}$ for $L$-smooth functions}

{We start with Proposition 
\ref{prop1mecl} showing an important property of
ME iterates for $L$-smooth functions.
We start recalling the definition of an
$L$-smooth function.}

\begin{definition} {We say that $f: \mathbb{R}^n \rightarrow
\mathbb{R}$ is an $L$-smooth function if $f$ is differentiable and for every $x,y \in \mathbb{R}^n$ we have
\begin{equation}\label{Lsmooth}
\|\nabla f(y)-\nabla f(x)\| \leq L \|y-x\|.
\end{equation}
}
\end{definition}

\begin{proposition}\label{prop1mecl} {Let $(x^k)$ be the sequence generated by ME. Assume 
$f$ is $\mu$-strongly convex and $L$-smooth. Then for every $k \geq 1$,}
\begin{equation}\label{ineqpropme}
f(x^{k+1})\leq f(x^k)-\frac{\mu}{2L^2}\|\nabla f(x^k)\|^2.
\end{equation}
\end{proposition}
\begin{proof} {Fix $k \geq 1$. 
Both in the case when 
$\nabla f(x^k)$ and 
$\nabla f(y^k)$ are
linearly independent
and in the case when 
$\nabla f(x^k)$ and 
$\nabla f(y^k)$ are
linearly dependent, by definition
of ME iterations, we have 
\begin{equation}\label{fxlxm}
f(x^{k+1}) \leq f\left( \frac{x^k+y^k}{2} \right).
\end{equation}
}
{We now show that
\begin{equation}\label{bdtkl}
t_k \ge \frac{2}{L}.
\end{equation}
Since $\nabla f$ is $L$-Lipschitz, the descent lemma gives
\[
f(x^k-t_k \nabla f(x^k))
\le
f(x^k)-t_k \|\nabla f(x^k)\|^2+\frac{L}{2}t_k^2\|\nabla f(x^k)\|^2 .
\]
Using
$f(x^k-t_k \nabla f(x^k))=f(x^k)$
and the condition
$t_k>0$, we deduce \eqref{bdtkl}.
}

{Next, strong convexity implies
\[
f\left(
\frac{x^k+y^k}{2}
\right)
\le
\frac12 f(x^k)+\frac12 f(y^k)
-\frac{\mu}{8}\|x^k-y^k\|^2.
\]
Using the fact that $f(y^k)=f(x^k)$,
we obtain 
\begin{equation}\label{fxkp1}
f\left(
\frac{x^k+y^k}{2}
\right)
\le
f(x^k)-\frac{\mu}{8}\|x^k-y^k\|^2.
\end{equation}
Plugging the relation
$
x^k-y^k=t_k \nabla f(x^k)$
into \eqref{fxkp1}, we obtain
\[
f\left(
\frac{x^k+y^k}{2}
\right)
\le
f(x^k)-\frac{\mu}{8}t_k^2\|\nabla f(x^k)\|^2.
\]
Using \eqref{bdtkl}, we deduce
\begin{equation}\label{finalxmlfgra}
f\left(
\frac{x^k+y^k}{2}
\right)
\le
f(x^k)-\frac{\mu}{2L^2}\|\nabla f(x^k)\|^2.
\end{equation}
Combining \eqref{fxlxm}
and \eqref{finalxmlfgra}
gives \eqref{ineqpropme}.
}
\end{proof}

\begin{theorem}[Linear convergence of ME]
\label{compME}
{Let $(x^k)$ be the sequence
generated by 
ME.
Assume that 
$f$ is $L$-smooth and $\mu$-strongly convex.
Define 
$$
0 \leq \eta=1- \frac{\mu^2}{L^2}<1.
$$
Then for every $k \geq 1$, we have
\begin{equation}\label{convfunc}
f(x^k)-f(x^*) \leq \eta^{k-1} (f(x^1)-f(x^*))
\end{equation}
and
\begin{equation}\label{convpoints}
\|x^k-x^*\|^2 \leq \frac{L}{\mu}  \eta^{k-1} \|x^1-x^*\|^2.
\end{equation}
}
\end{theorem}
\begin{proof}
{By Proposition \ref{prop1mecl} we have for all $k \geq 1$,
\begin{equation}\label{ineqpropfundi}
f(x^{k+1})\leq f(x^k)-\frac{\mu}{2L^2}\|\nabla f(x^k)\|^2.
\end{equation}
A standard consequence of $\mu$-strong convexity is
the Polyak-Lojasiewicz
inequality}
{
\[
\|\nabla f(x)\|^2
\ge
2\mu\bigl(f(x)-f(x^*)\bigr)
\]
for every $x \in \mathbb{R}^n$. Plugging this inequality written for $x=x^k$ into \eqref{ineqpropfundi} gives
}
{
\[
\begin{aligned}
f(x^{k+1})-f(x^*)
&\le
f(x^k)-f(x^*)
-\frac{\mu}{2L^2}
\|\nabla f(x^k)\|^2 \\
&\le
f(x^k)-f(x^*)
-\frac{\mu^2}{L^2}
\Bigl(f(x^k)-f(x^*)\Bigr).
\end{aligned}
\]
We therefore have
}
{
\begin{equation}\label{convfunc0}
f(x^{k+1})-f(x^*) \leq \left(1- \frac{\mu^2}{L^2} \right)  \Bigl(f(x^{k})-f(x^*)\Bigr),
\end{equation}
}
{
which immediately implies \eqref{convfunc} by induction.
Also observe that since
$
0<\mu\le L$, the convergence rate 
$\eta=1-\mu^2/L^2$ satisfies
$0 \leq \eta <1$.
We now prove \eqref{convpoints}. 
By strong convexity of $f$ and using the optimality condition $\nabla f(x^*)=0$, we have
}
{
\begin{equation}\label{sconvf}
f(x^k)-f(x^*) \geq \langle \nabla f(x^*),x-x^* \rangle + \frac{\mu}{2}\|x^k-x^*\|^2
=\frac{\mu}{2}\|x^k-x^*\|^2.
\end{equation}
Using again that $f$ is $L$-smooth and
the optimality condition $\nabla f(x^*)=0$, we
obtain 
}
{
\begin{equation}\label{convlgrad}
f(x_1) - f(x^*) \leq   \langle \nabla f(x^*),x_1-x^* \rangle + \frac{L}{2}\|x_1-x^*\|^2
=\frac{L}{2}\|x_1-x^*\|^2.
\end{equation}
Combining \eqref{sconvf},  \eqref{convfunc0},  and \eqref{convlgrad} gives
}
{
$$
\frac{\mu}{2}\|x^k-x^*\|^2 
\leq f(x^k)-f(x^*) \leq \eta^{k-1}(f(x_1)-f(x^*)) \leq \frac{L}{2}\eta^{k}\|x_1-x^*\|^2
$$
which gives \eqref{convpoints}.}
\end{proof}

\section{Numerical experiments}\label{sec:num} 

{In this section, we test the following four variants of ME to minimize a $\ell_2$-regularized logistic loss function:
\begin{itemize}
\item the first variant is ME-midpoint described in Section \ref{sec:algo};
\item the second variant is ME-linesearch described in Section \ref{sec:algo};
\item the third variant, denoted by ME-optimize, is ME as described 
in Section \ref{sec:algo} where the minimization of $f$
on the half-line $\{z^k+vd^k:v \geq 0\}$ is performed using gradient descent with line search
at every itertion (we use Julia function optimize of optimization package Optim);
\item the fourth variant, denoted by ME-optimize-BB, is ME as described  
in Section \ref{sec:algo} where the minimization of $f$
on the half-line $\{z^k+vd^k:v \geq 0\}$ is performed using (our own implementation of)
Barzilai-Borwein method with short steps.
\end{itemize}
}

{Recall that our convergence analysis of ME, given in Theorems
\ref{convME1} and \ref{compME}, applies to these variants of ME.}

{We compare these variants of ME with gradient descent with
exact line search, Nesterov method \cite{nesterov1983}, and the
method from \cite{MEbellocruz} denoted by ME-Bello Cruz in what follows.
}

{The $\ell_2$-regularized logistic loss 
considered, for instance, in \cite{MEbellocruz}, is given by
\begin{equation}
f(x) = \frac{1}{m}\sum\limits_{i=1}^m \log(1 + e^{-b_ia_i^{T}x}) + \frac{\mu}{2} \left \|x \right\|^2,
\end{equation}
where $a_i \in \mathbb{R}^n$,
$b_i \in \{-1,+1\}$ are drawn from standard Gaussian distribution
and $\mu>0$. For this function
$f$, relation \eqref{Lsmooth} holds with
$L=\frac{1}{4m}\sum_{i=1}^m \|a_i\|^2+\mu$ and parameter
$\mu$ is chosen as in \cite{MEbellocruz} fixing
$\kappa=L/\mu$ and solving for $\mu$.}

 {All optimization methods were implemented in Julia 1.13 and all experiments were run on a cloud virtual machine with Ubuntu 22.04 machine with 8 GB of RAM and an Intel Xeon processor with 2.5 GHz. The stopping criterion is $\left \| \nabla f(x^k) \right \| \leq 10^{-6}$.
 The maximal number of iterations for all methods is
 1000 and we use a 10 minutes upper bound for the running time
 of all methods. Our code is available on github at {\url{https://github.com/vguigues/Ellipcenter-method/tree/main/Julia\_Convex}}}
 
 {We provide below the pseudo-code used to approximate the root $t_k$ that satisfies $f(x^k - t_k \nabla f(x^k))=f(x^k)$. This pseudo-code performs a dichotomy 
using a relative gap for the stopping criterion. In this pseudo-code,
to alleviate notation, we set
$x=x^k$.
}

\if{
\noindent\rule[0.5ex]{1\columnwidth}{1pt}
ME -- Absolute gap.\\
\noindent\rule[0.5ex]{1\columnwidth}{1pt}
{\bf Inputs:} Point $x\in\mathbb{R}^n$, 
$\mathrm{gradient}$, $f_{\mathrm{value}}$, and $\mathrm{tol}>0$.\\

\noindent {\textbf{Step 1.}} Set
\(
t_L \gets 0, \qquad
t_R \gets 2,
\)
and compute
\(
f_L \gets f_{\mathrm{value}},
f_R \gets \operatorname{valuef}
(x+t_R\,\mathrm{gradient}).
\)

\noindent {\textbf{Step 2.}} While
$f_R<f_{\mathrm{value}}$, perform
\[
t_R \gets 2t_R
\]
and compute
\[
f_R \gets
\operatorname{valuef}(x+t_R\,\mathrm{gradient}).
\]
End while.

\noindent {\textbf{Step 3.}} Set
\(
t\gets\frac{t_L+t_R}{2}
\)
and compute
\(
f\gets\operatorname{valuef}(x+t\,\mathrm{gradient}).
\)

\noindent {\textbf{Step 4.}} If
\(
|f-f_{\mathrm{value}}|
\leq \mathrm{tol},
\)
stop and return $t$. Otherwise, go to Step 5.

\noindent {\textbf{Step 5.}} If
$f<f_{\mathrm{value}}$, set
\(
t_L\gets t,
f_L\gets f.
\)
Otherwise, set
\(
t_R\gets t,
f_R\gets f.
\)
End if.\\
Perform
\(
t\gets\frac{t_L+t_R}{2}
\)
and compute
\(
f\gets\operatorname{valuef}(x+t\,\mathrm{gradient}).
\)
Go to Step 4.\\
\noindent\rule[0.5ex]{1\columnwidth}{1pt}

}\fi

\noindent\rule[0.5ex]{1\columnwidth}{1pt}
Dichotomy to compute $t_k$\\
\noindent\rule[0.5ex]{1\columnwidth}{1pt}
{\bf Inputs:} Point $x\in\mathbb{R}^n$, 
$\nabla f(x)$, $f(x)$, and $\mathrm{tol}>0$.\\

\noindent {\textbf{Step 1.}} Set
\(
t_L \gets 0, \; t_R \gets 2,
\)
and compute
\(
f_L \gets f(x),
f_R \gets f(x-t_R\,\nabla f(x)).
\)

\noindent {\textbf{Step 2.}} While
$f_R< f(x)$, perform
\[
t_R \gets 2t_R
\]
and compute
\[
f_R \gets
f (x-t_R\,\nabla f(x)).
\]
End while.

\noindent {\textbf{Step 3.}} Set
\(
t\gets\frac{t_L+t_R}{2}
\)
and compute
\(
f\gets f(x-t\,\nabla f(x)).
\)

\noindent {\textbf{Step 4.}} If
\(
|f-f(x)|
\leq \mathrm{tol}\max(|f(x) |,1),
\)
stop and return $t$. Otherwise, go to Step 5.

\noindent {\textbf{Step 5.}} If
$f<f(x)$, set
\(
t_L\gets t,
f_L\gets f.
\)
Otherwise, set
\(
t_R\gets t,
f_R\gets f.
\)
End if.\\
Perform
\(
t\gets\frac{t_L+t_R}{2}
\)
and compute
\(
f\gets f(x-t\,\nabla f(x)).
\)
Go to Step 4.\\
\noindent\rule[0.5ex]{1\columnwidth}{1pt}

{In this pseudo code which approximates
$t_k$ by dichotomy we take
tol=$10^{-6}$, i.e., the dichotomy stops when
$|f-f(x^k)|
\leq 10^{-6}\max(|f(x^k)|,1)$.}

{We run each instance 10 times, and due to Julia's JIT compiler, we disconsider the run time of the first execution of each method.}

{We run all methods for 
(i) $n=500$, $\kappa=100$,
(ii) $n=1000$, $\kappa=500$, and 
(iii) $n=2000$, $\kappa=1000$.
For these parameters, we report in Table 
\ref{table_experiment1} the mean optimal
value at termination (in column 
$f(x^k)$), the mean number of
function evaluations (in column \# $f$ eval.),
the mean number of gradient evaluations 
(in column \# $\nabla f$ eval.), 
the mean norm of the gradient at termination
(in column 
$\left\|\nabla f(x^{k})\right\|_2$), 
the mean and standard deviation of the number
of iterations (in column \# iters)
and the mean and standard deviation of the
CPU time (in column time).}

{For $n=500$ and $\kappa=100$ all methods except ME-Bello Cruz satisfy the stopping criterion 
$\left \| \nabla f(x^k) \right \| \leq 10^{-6}$
in less than 1000 iterations.
Our variants ME-optimize and ME-BB require much less
iterations, are the quickest with ME-optimize-BB
being the quickest by far.}

{For $n=1000$ and $\kappa=500$
and
$n=2000$ and $\kappa=1000$, all methods except ME-Bello Cruz
satisfy the stopping criterion 
$\left \| \nabla f(x^k) \right \| \leq 10^{-6}$
in less than 1000 iterations.
Our variants ME-optimize and ME-BB require much less
iterations with ME-optimize-BB
being the quickest. The competitors 
GD-exact, ME-Bello Cruz, and Nesterov are much
slower.}

\begin{table}[H]
\centering
\caption{{Logistic regression with $n=500,1000,2000$ and $\kappa=100,500,1000$. The table shows the mean objective values, mean number of function and gradient evaluations, the mean final norm of the gradient, the mean and standard deviation of the run time (in seconds) over 10 trials, with the first trial run time discarded.}}
\label{table_experiment1}
\scalebox{0.8}{
\begin{tabular}{lrrrrrrrr}
\toprule
Method & n & $\kappa$ & $f(x^{k})$ & \# $f$ eval. & \# $\nabla f$ eval. & $\left\|\nabla f(x^{k})\right\|_2$ & \# iters & time \\
\midrule
GD-exact & 500 & 100.0 & 0.563 & 1204.3 & 1204.3 & $5.40 \times 10^{-7}$ & $11.0 \pm 0.0$ & $0.186 \pm 0.060$ \\
ME-bello-cruz & 500 & 100.0 & 0.563 & 1003668.6 & 1002566.0 & $5.46 \times 10^{-4}$ & $1000.0 \pm 0.0$ & $125.573 \pm 3.035$ \\
ME-linesearch & 500 & 100.0 & 0.563 & 164.4 & 28.5 & $7.06 \times 10^{-7}$ & $9.2 \pm 0.4$ & $0.010 \pm 0.002$ \\
ME-midpoint & 500 & 100.0 & 0.563 & 74.0 & 23.0 & $4.89 \times 10^{-7}$ & $11.0 \pm 0.0$ & $0.007 \pm 0.002$ \\
ME-optimize & 500 & 100.0 & 0.563 & 168.6 & 143.9 & $4.72 \times 10^{-7}$ & $5.0 \pm 0.0$ & $0.022 \pm 0.005$ \\
ME-optimize-BB & 500 & 100.0 & 0.563 & 41.0 & 27.3 & $5.42 \times 10^{-7}$ & $5.1 \pm 0.3$ & $0.004 \pm 0.001$ \\
Nesterov & 500 & 100.0 & 0.563 & 993.3 & 993.3 & $9.95 \times 10^{-7}$ & $992.3 \pm 6.7$ & $0.086 \pm 0.011$ \\
GD-exact & 1000 & 500.0 & 0.464 & 2939.5 & 2939.5 & $7.28 \times 10^{-7}$ & $17.1 \pm 0.3$ & $1.230 \pm 0.330$ \\
ME-bello-cruz & 1000 & 500.0 & 0.464 & 1004532.4 & 1002416.6 & $8.20 \times 10^{-4}$ & $1000.0 \pm 0.0$ & $402.466 \pm 13.384$ \\
ME-linesearch & 1000 & 500.0 & 0.464 & 206.8 & 36.7 & $6.29 \times 10^{-7}$ & $12.1 \pm 1.4$ & $0.042 \pm 0.006$ \\
ME-midpoint & 1000 & 500.0 & 0.464 & 105.6 & 32.0 & $6.13 \times 10^{-7}$ & $15.5 \pm 0.8$ & $0.025 \pm 0.005$ \\
ME-optimize & 1000 & 500.0 & 0.464 & 295.6 & 263.0 & $5.53 \times 10^{-7}$ & $7.7 \pm 0.5$ & $0.105 \pm 0.021$ \\
ME-optimize-BB & 1000 & 500.0 & 0.464 & 54.2 & 42.8 & $8.10 \times 10^{-7}$ & $7.2 \pm 0.4$ & $0.020 \pm 0.003$ \\
Nesterov & 1000 & 500.0 & 0.464 & 1000.0 & 1000.0 & $2.69 \times 10^{-4}$ & $999.0 \pm 0.0$ & $0.254 \pm 0.037$ \\
GD-exact & 2000 & 1000.0 & 0.468 & 2399.6 & 2399.6 & $7.41 \times 10^{-7}$ & $17.0 \pm 0.0$ & $5.895 \pm 2.397$ \\
ME-bello-cruz & 2000 & 1000.0 & 0.468 & 249018.6 & 247934.0 & $1.39 \times 10^{-3}$ & $247.8 \pm 18.2$ & $601.052 \pm 0.924$ \\
ME-linesearch & 2000 & 1000.0 & 0.468 & 206.7 & 36.9 & $5.50 \times 10^{-7}$ & $12.3 \pm 0.8$ & $0.238 \pm 0.064$ \\
ME-midpoint & 2000 & 1000.0 & 0.468 & 106.5 & 32.2 & $4.87 \times 10^{-7}$ & $15.6 \pm 1.0$ & $0.159 \pm 0.055$ \\
ME-optimize & 2000 & 1000.0 & 0.468 & 950.4 & 917.0 & $6.73 \times 10^{-7}$ & $7.6 \pm 0.7$ & $0.841 \pm 0.340$ \\
ME-optimize-BB & 2000 & 1000.0 & 0.468 & 55.9 & 44.3 & $6.48 \times 10^{-7}$ & $7.4 \pm 0.5$ & $0.117 \pm 0.032$ \\
Nesterov & 2000 & 1000.0 & 0.468 & 1000.0 & 1000.0 & $1.02 \times 10^{-3}$ & $999.0 \pm 0.0$ & $1.711 \pm 0.313$ \\
\bottomrule
\end{tabular}}
\end{table}

\section{Conclusion}\label{sec:conc}

\par {In this work, we introduced the Method of Ellipcenters
for minimizing a strongly convex differentiable objective
function.}

\par {As a future work, ME could be extended for nondifferentiable 
and/or constrained optimization problems.}

\par {We also intend to incorporate relaxation and moments in
ME to accelerate the method.}\\

% \noindent\rule[0.5ex]{1\columnwidth}{1pt}
% ME -- Relative gap.\\
% \noindent\rule[0.5ex]{1\columnwidth}{1pt}
% {\bf Inputs:} Point $x\in\mathbb{R}^n$, 
% $\mathrm{gradient}$, and $f_{\mathrm{value}}$.\\

% \noindent {\textbf{Step 1.}} Set
% \(
% t_L \gets 0, \qquad
% t_R \gets 2,
% \)
% and compute
% \(
% f_L \gets f_{\mathrm{value}},
% f_R \gets \operatorname{valuef}
% (x+t_R\,\mathrm{gradient}).
% \)

% \noindent {\textbf{Step 2.}} While
% $f_R<f_{\mathrm{value}}$, perform
% \[
% t_R \gets 2t_R
% \]
% and compute
% \[
% f_R \gets
% \operatorname{valuef}(x+t_R\,\mathrm{gradient}).
% \]
% End while.

% \noindent {\textbf{Step 3.}} Set
% \(
% t\gets\frac{t_L+t_R}{2}
% \)
% and compute
% \(
% f\gets\operatorname{valuef}(x+t\,\mathrm{gradient}).
% \)

% \noindent {\textbf{Step 4.}} If
% \(
% \frac{|f-f_{\mathrm{value}}|}
% {|f_{\mathrm{value}}|}
% \leq 10^{-4},
% \)
% stop and return $t$. Otherwise, go to Step 5.

% \noindent {\textbf{Step 5.}} If
% $f<f_{\mathrm{value}}$, set
% \(
% t_L\gets t,
% f_L\gets f.
% \)
% Otherwise, set
% \(
% t_R\gets t,
% f_R\gets f.
% \)
% End if.\\
% Perform
% \(
% t\gets\frac{t_L+t_R}{2}
% \)
% and compute
% \(
% f\gets\operatorname{valuef}(x+t\,\mathrm{gradient}).
% \)
% Go to Step 4.\\
% \noindent\rule[0.5ex]{1\columnwidth}{1pt}

\par {\textbf{Funding.}} Roger Behling was partially supported by Conselho Nacional de Desenvolvimento Científico e Tecnológico (CNPq) Grant 309458/2025-0 and Fundacao de Amparo a Pesquisa e Inovacao do Estado de Santa Catarina (FAPESC) Grant 2024TR002238. Vincent Guigues was partially supported by CNPq and acknowledges the support of CNPq grant 305263/2023-4. Eduarda Ferreira Zanatta was supported by Fundação de Amparo à Pesquisa do Estado de São Paulo (FAPESP) grant 2025/21077-9.\\

\par {\textbf{Competing interests.}} The authors have no competing interests to declare.

\if{
Because $f(x^*)< f(x^{k+1})<f(x^k)$ for all $k$, $\lim_{k \rightarrow \infty} f(x^k) = f(\bar{x})$ for some $\bar{x} \in \mathbb{R}^n$. Therefore, since $$(f(x^{k+1})-f(x^*))-(f(x^0)-f(x^*))=\sum_{k=0}^k (f(x^n)-f(z^n))+(f(z^n)-f(x^{n+1}))$$ we have that $$\sum_{k=0}^{\infty}(f(x^n)-f(z^n))+(f(z^n)-f(x^{n+1}))< \infty$$
so $\lim_{n \rightarrow \infty} (f(x^n)-f(z^n))=0$.
}\fi


\begin{thebibliography}{99}

\bibitem{MEbellocruz}
Bello-Cruz, Y.  
\newblock The Method of Ellipcenters for Strongly Convex Functions.  
\newblock {\em arXiv}, 2026.

\bibitem{ourME}
Behling, R. and Correa R. and Ferreira E. and Guigues, V.  
\newblock Introducing the method of ellipcenters, a new first order technique for unconstrained optimization.  
\newblock {\em Optimization Methods and Software}, 2026.

\bibitem{BarzilaiBorwein1988}
Barzilai, J. and Borwein, J. M.  
\newblock Two-Point Step Size Gradient Methods.  
\newblock {\em IMA Journal of Numerical Analysis}, 8(1):141--148, 1988.

\bibitem{fista}
Beck, A. and Teboulle, M.  
\newblock A Fast Iterative Shrinkage-Thresholding Algorithm
 for Linear Inverse Problems   
\newblock {\em SIAM J. IMAGING SCIENCES}, 2 (1):183-202, 2009.

\bibitem{Cauchy1847}
Cauchy, A.-L.  
\newblock Méthode générale pour la résolution des systèmes d'équations simultanées.  
\newblock {\em Comptes Rendus Hebdomadaires des Séances de l'Académie des Sciences}, 25:536--538, 1847.

\bibitem{Fletcherroger87}
Fletcher, R.  
\newblock Practical Methods of Optimization (2nd ed.)  
\newblock {\em New York: John Wiley \& Son, 1987}.

\bibitem{FletcherReeves1964}
Fletcher, R. and Reeves, C. M.  
\newblock Function Minimization by Conjugate Gradients.  
\newblock {\em The Computer Journal, 7(2):149--154, 1964}.

\bibitem{gonzagakaras}
Gonzaga, C. and Karas. E.
\newblock Fine tuning Nesterov’s steepest descent algorithm for differentiable convex programming
\newblock {\em Mathematical Programming}, 138:141--166,2013.

\bibitem{grimmer2023optimal}
Grimmer, B.
\newblock On optimal universal first-order methods for minimizing heterogeneous sums.  
\newblock {\em Optimization Letters, 1--19, 2023}. 

\bibitem{Khachiyan1979}
Khachiyan, L.G.  
\newblock A polynomial algorithm in linear programming.  
\newblock {\em Soviet Mathematics Doklady}, 20:191--194, 1979.

\bibitem{lan2015bundle}
Lan, G.
\newblock Bundle-level type methods uniformly optimal for smooth and nonsmooth convex optimization  
\newblock {\em Mathematical Programming}, 149(1-2):1--45,2015.

\bibitem{liang2021average}
Liang, J. and Monteiro, R. D. C.
\newblock An average curvature accelerated composite gradient method for nonconvex smooth composite optimization problems  
\newblock {\em SIAM Journal on Optimization}, 31(1):217--243,2021.

\bibitem{liang2023average}
Liang, J. and Monteiro, R. D. C.
\newblock Average curvature FISTA for nonconvex smooth composite optimization problems 
\newblock {\em Computational Optimization and Applications}, 86(1):275--302,2023.

\bibitem{mishchenko2020adaptive}
Mishchenko, K. and Malitsky, Y.
\newblock Adaptive gradient descent without descent 
\newblock {\em 37th International Conference on Machine Learning (ICML 2020)},2020.

\bibitem{nesterov1983}
Nesterov, Y.  
\newblock A method of solving a convex programming problem with convergence rate $O(1/k^2)$.  
\newblock {\em Soviet Mathematics Doklady}, 27(2):372--376, 1983.

\bibitem{nesterov2015universal}
Nesterov, Y.
\newblock Universal gradient methods for convex optimization problems    
\newblock {\em Mathematical Programming}, 152(1):381--404,2015.

 
\bibitem{renegar2022simple}
Renegar, J. and Grimmer, B.
\newblock A simple nearly optimal restart scheme for speeding up first-order methods  
\newblock {\em Foundations of Computational Mathematics}, 22(1):211--256,2022. 


\bibitem{Wolfe1969}
Wolfe, P.  
\newblock Convergence Conditions for Ascent Methods.  
\newblock {\em SIAM Review}, 11(2):226--235, 1969.

\bibitem{zhou2024adabb}
Zhou, D. and Ma, S. and Yang, J.
\newblock Ada{BB}: Adaptive {B}arzilai-{B}orwein method for convex optimization. 
\newblock {\em Mathematics of Operations Research, 51(1):715--745}, 2025.


\end{thebibliography}
\end{document}